\documentclass[11pt,a4paper]{article}

\usepackage[T1]{fontenc}
\usepackage[utf8]{inputenc}
\usepackage{amsmath,amsthm,amssymb,amsfonts}
\usepackage{mathtools}
\usepackage{geometry}
\usepackage{hyperref}
\usepackage{enumitem}
\usepackage{microtype}
\usepackage{booktabs}

\hypersetup{hidelinks}
\setlist{
  itemsep=2pt,
  topsep=4pt,
  parsep=0pt,
  partopsep=0pt
}

\makeatletter
\renewenvironment{abstract}{%
  \par\vspace{0.75em}%
  \begingroup
  \small
  \noindent\textbf{\abstractname.}\quad\ignorespaces
}{%
  \par\endgroup\vspace{1.1em}%
}
\makeatother

\newcommand{\Pn}{\mathcal{P}_n}

\newcommand{\Cay}{\operatorname{Cay}}
\newcommand{\Aut}{\operatorname{Aut}}
\newcommand{\Stab}{\operatorname{Stab}}
\newcommand{\eval}{\operatorname{eval}}
\newcommand{\id}{\operatorname{id}}
\newcommand{\ord}{\operatorname{ord}}
\newcommand{\Comp}{\kappa}

\newcommand{\calT}{\mathcal{T}}

\newtheorem{theorem}{Theorem}[section]
\newtheorem*{theorem*}{Theorem}
\newtheorem{lemma}[theorem]{Lemma}
\newtheorem{proposition}[theorem]{Proposition}
\newtheorem{corollary}[theorem]{Corollary}

\theoremstyle{definition}
\newtheorem{definition}[theorem]{Definition}
\newtheorem{example}[theorem]{Example}
\newtheorem{problem}[theorem]{Problem}

\theoremstyle{remark}
\newtheorem{remark}[theorem]{Remark}

\title{\textbf{The Left-Regular Stabilizer of Zaks' Hamiltonian Cycle in the Pancake Graph}}

\author{
Or-Hai Benjo\thanks{\href{mailto:ksamim12321@gmail.com}{\texttt{ksamim12321@gmail.com}}}
\and
Yehonathan Sharvit\thanks{\href{mailto:viebel@gmail.com}{\texttt{viebel@gmail.com}}}
}

\date{}

\begin{document}

\maketitle

% ==================================================
\begin{abstract}
Let $\Pn=\Cay(S_n,\{r_2,\ldots,r_n\})$ be the pancake graph, with prefix reversals acting on the right. Conjugating Zaks' suffix-reversal permutation Gray code by the full reversal
gives a distinguished Hamiltonian cycle \(Z_n\) in \(\Pn\). We determine the
stabilizer of this particular cycle under the left regular action of \(S_n\). If
\[
\rho=r_{n-1}r_n=[n,1,2,\ldots,n-1],
\]
then, for every \(n\ge3\),
\[
\Stab_{L(S_n)}(Z_n)=\langle L_\rho,L_{r_n}\rangle\cong D_n,
\]
where \(D_n\) denotes the dihedral group of order \(2n\). The inclusion
\(\supseteq\) follows from the recursive block decomposition $W_n=(W_{n-1}r_n)^{n-1}W_{n-1}$ and from the palindromy \(W_n^R=W_n\). The reverse inclusion follows from a
general cyclic-order rigidity lemma: if a Hamiltonian cycle on a finite group
is invariant under \(L_a\), with \(\ord(a)\ge3\), then every left translation
preserving the same cycle conjugates \(a\) to \(a\) or \(a^{-1}\). For
\(n\ge5\), Deng--Zhang's automorphism theorem gives the same stabilizer inside
\(\Aut(\Pn)\); the exceptional ranks are handled separately. We also compute
the compression factor of \(Z_n\): it is \(n\) for \(n\ge4\) and \(6\) for
\(n=3\).
\end{abstract}

% ==================================================
\section{Introduction}

The pancake graph is the Cayley graph of the symmetric group generated by
prefix reversals. Its vertices are the permutations of
\([n]=\{1,\ldots,n\}\), and two vertices are adjacent precisely when one is
obtained from the other by reversing an initial segment. Hamiltonian cycles in
this graph are cyclic permutation Gray codes whose transition set is restricted
to prefix reversals. The graph appears naturally both in pancake sorting and in
the theory of symmetric interconnection networks; see, for example,
\cite{akers-krishnamurthy,gates-papadimitriou,lakshmivarahan-jwo-dhall,
kanevsky-feng,deng-zhang}.

The present paper is not concerned with Hamiltonicity itself. Explicit
permutation-generation algorithms already supply Hamiltonian cycles in
\(\Pn\). Thus the contribution is not a new Hamiltonicity theorem for pancake
graphs, but an exact stabilizer computation for a canonical Hamiltonian cycle.
Our object is the stabilizer of one distinguished cycle: the cycle obtained
from Zaks' recursive suffix-reversal Gray code \cite{zaks} after conjugating
suffix reversals into prefix reversals. We prove that the evident dihedral
symmetries forced by the recursion are the only left-regular symmetries of
this cycle. Our object is the stabilizer of one distinguished cycle: the cycle
obtained from Zaks' recursive suffix-reversal Gray code \cite{zaks} after
conjugating suffix reversals into prefix reversals. We prove that the evident
dihedral symmetries forced by the recursion are the only left-regular
symmetries of this cycle.

Zaks constructed an ordering of the $n!$ permutations in which consecutive permutations differ by reversing a suffix. Conjugation by the full reversal transforms suffix reversals into prefix reversals. With the right-Cayley convention fixed below, this gives a Hamiltonian path in $\Pn$, from $\id$ to $r_n$; adjoining the edge $\{r_n,\id\}$ gives a Hamiltonian cycle, denoted $Z_n$.

For $g\in S_n$, let $L_g:S_n\to S_n$ denote left multiplication, $L_g(\pi)=g\pi$, and put $\rho=r_{n-1}r_n=[n,1,2,\ldots,n-1]$.

\begin{theorem*}[Main theorem]
For every $n\ge3$,
\[
\Stab_{L(S_n)}(Z_n)=\langle L_\rho,L_{r_n}\rangle\cong D_n,
\]
where $D_n$ is the dihedral group of order $2n$.
\end{theorem*}
The proof separates the construction-specific symmetries from a general
group-theoretic obstruction. The recursion $W_n=(W_{n-1}r_n)^{n-1}W_{n-1}$ decomposes \(Z_n\) into \(n\) consecutive blocks of length \((n-1)!\), whose
initial vertices are
\[
\id,\rho,\rho^2,\ldots,\rho^{n-1}.
\]
Thus \(L_\rho\) rotates the block system. The palindromy \(W_n^R=W_n\),
together with \(\eval(W_n)=r_n\), gives a reflection \(L_{r_n}\). These two
visible symmetries generate a dihedral subgroup of order \(2n\).

The upper bound is independent of the pancake generating set. If a Hamiltonian
cycle on a finite group \(G\) is invariant under \(L_a\), with
\(\ord(a)\ge3\), then every left translation preserving the same cyclic order
conjugates \(a\) to \(a\) or \(a^{-1}\). Applied to the \(n\)-cycle
\(\rho\in S_n\), this confines the left-regular stabilizer to the dihedral
normalizer of \(\rho\), which has order \(2n\).

For \(n\ge5\), Deng--Zhang's automorphism theorem gives \(\Aut(\Pn)=L(S_n)\), so the same dihedral group is the full graph-automorphism stabilizer of \(Z_n\). We also compute the compression factor of \(Z_n\) in the sense of Gregor, Merino, and M\"utze \cite{gregor-merino-mutze}: it is \(n\) for \(n\ge4\), while the exceptional graph \(\mathcal P_3\) gives compression factor \(6\). This statement concerns the particular cycle \(Z_n\), not the graph invariant \(\Comp(\Pn)\).

The paper first fixes the prefix-reversal form of Zaks' construction, then proves
the visible dihedral symmetries, the cyclic-order rigidity bound, the full
automorphism stabilizer, and finally the compression and quotient consequences.

% ==================================================
\section{Conventions and Zaks' Cycle}
\label{sec:conventions}

We use one-line notation \(\pi=[\pi(1),\ldots,\pi(n)]\) for permutations in \(S_n\). Composition is from right to left: \((\pi\sigma)(i)=\pi(\sigma(i))\). Thus right multiplication acts on positions. In particular, if a permutation \(\sigma\) reverses the first \(k\) positions, then the one-line notation of \(\pi\sigma\) is obtained from that of \(\pi\) by reversing the first \(k\) entries.

\begin{definition}[Prefix reversals]
For \(2\le k\le n\), define \(r_k=[k,k-1,\ldots,2,1,k+1,\ldots,n]\in S_n\). Thus \(r_k\) reverses the first \(k\) positions and fixes all positions greater than \(k\). Each \(r_k\) is an involution.
\end{definition}

\begin{definition}[The pancake graph]
Let \(R_n=\{r_2,\ldots,r_n\}\). The pancake graph is the undirected right Cayley graph \(\Pn=\Cay(S_n,R_n)\), with vertex set \(S_n\) and edge set \(E(\Pn)= \bigl\{\{\pi,\pi r_k\}:\pi\in S_n,\ 2\le k\le n\bigr\}\).
\end{definition}

Since the generators $r_k$ are involutions and are distinct, $\Pn$ is a simple undirected graph. For every $g\in S_n$, left multiplication $L_g(\pi)=g\pi$ is an automorphism of $\Pn$, because
\[
L_g(\pi r_k)=g(\pi r_k)=(g\pi)r_k=L_g(\pi)r_k.
\]
We write $L(S_n)=\{L_g:g\in S_n\}$.

We shall use the canonical edge label \(\lambda(\{\pi,\pi r_k\})=r_k\).
This label is well-defined. Indeed, if $\{\pi,\pi r_k\}=\{\sigma,\sigma r_\ell\},$
then the quotient of the two endpoints is \(r_k\) in one orientation and
\(r_\ell\) in the other. Since \(r_k^{-1}=r_k\) and \(r_\ell^{-1}=r_\ell\),
it follows that \(r_k=r_\ell\), and hence \(k=\ell\). In particular, every
left translation preserves \(\lambda\).

\begin{definition}[Suffix reversals]
Let \(J_n=r_n=[n,n-1,\ldots,1]\). For \(2\le k\le n\), define
\[
t_k=[1,2,\ldots,n-k,n,n-1,\ldots,n-k+1].
\]
Thus right multiplication by \(t_k\) reverses the last \(k\) entries of a
one-line permutation.
\end{definition}

Under the usual identification of a word \(p_1p_2\cdots p_n\) with the
one-line permutation \([p_1,\ldots,p_n]\), Zaks' operation of reversing the
last \(k\) entries is exactly right multiplication by \(t_k\) under the
composition convention fixed above.

\begin{lemma}[Prefix--suffix conjugacy]
\label{lem:prefix-suffix-conjugacy}
For every \(2\le k\le n\), \(J_nt_kJ_n=r_k\). Consequently, conjugation by \(J_n\) transforms a suffix-reversal walk with transition labels \(t_{i_1},\ldots,t_{i_m}\) into a prefix-reversal walk with transition labels \(r_{i_1},\ldots,r_{i_m}\).
\end{lemma}

\begin{proof}
The permutation \(J_n\) reverses the order of positions. Hence conjugation by \(J_n\) carries the set of last \(k\) positions onto the set of first \(k\) positions and transports the reversal of that suffix to the reversal of the corresponding prefix. Equivalently, for \(1\le i\le n\),
\[
(J_nt_kJ_n)(i)
=
\begin{cases}
k+1-i, & 1\le i\le k,\\
i, & k<i\le n,
\end{cases}
\]
which is exactly \(r_k\). If \(\pi_{j+1}=\pi_jt_k\), then
\[
J_n\pi_{j+1}J_n
=
J_n(\pi_jt_k)J_n
=
(J_n\pi_jJ_n)(J_nt_kJ_n)
=
(J_n\pi_jJ_n)r_k.
\]
Thus the conjugated walk has the asserted prefix-reversal transition labels.
\end{proof}

For a word \(W=s_1s_2\cdots s_m\) in the alphabet \(R_n\), define \(\eval(W)=s_1s_2\cdots s_m\in S_n\). The word \(W\), read from the identity, generates the vertex sequence \(u_0=\id,\ u_i=s_1s_2\cdots s_i\quad(1\le i\le m)\). We also write \(W^R=s_m\cdots s_2s_1\) for the reversed word.

Zaks' suffix-reversal sequence is recursively defined by \(s_2=2,\ s_n=(s_{n-1}n)^{n-1}s_{n-1} \quad(n\ge3)\). Here \(s_n\) is a sequence of suffix lengths. Replacing each integer \(k\) by the suffix reversal \(t_k\) gives a word in the generators \(t_2,\ldots,t_n\). We use the following theorem of Zaks in precisely this form.

\begin{theorem}[Zaks]
\label{thm:zaks-original}
For every \(n\ge2\), the suffix-reversal word determined by \(s_n\) has length \(n!-1\), and successive right multiplication by the corresponding suffix reversals, starting at \(\id\), visits every element of \(S_n\) exactly once.
\end{theorem}

\begin{definition}[The prefix Zaks word]
Define \(W_2=r_2\). For \(n\ge3\), define recursively \(W_n=(W_{n-1}r_n)^{n-1}W_{n-1}\). In this formula \(W_{n-1}\) is viewed as a word in the embedded generators \(r_2,\ldots,r_{n-1}\in S_n\), each of which fixes the \(n\)-th position.
\end{definition}
\begin{lemma}
\label{lem:rho-cycle-early}
Let $\rho_n=r_{n-1}r_n\in S_n$. Then $\rho_n=[n,1,2,\ldots,n-1]$, and $\rho_n$ has order $n$.
\end{lemma}

\begin{proof}
Rightmost factors act first. For $i=1$, $r_n(1)=n$, and $r_{n-1}$ fixes $n$, so $\rho_n(1)=n$. For $2\le i\le n$, one has $r_n(i)=n+1-i\le n-1$, and hence
\[
\rho_n(i)=r_{n-1}(n+1-i)=i-1.
\]
Thus $\rho_n=[n,1,2,\ldots,n-1]$, an $n$-cycle.
\end{proof}

\begin{lemma}[Endpoint]
\label{lem:endpoint}
For every $n\ge2$, $\eval(W_n)=r_n$.
\end{lemma}

\begin{proof}
The case \(n=2\) is immediate. Assume \(n\ge3\) and
\(\eval(W_{n-1})=r_{n-1}\), with \(S_{n-1}\) embedded in \(S_n\)
by fixing \(n\). Since \(\rho_n=r_{n-1}r_n\),
\[
\eval(W_n)
=
(r_{n-1}r_n)^{n-1}r_{n-1}
=
\rho_n^{n-1}r_{n-1}.
\]
Moreover, \(r_{n-1}=\rho_n r_n\). Therefore $\eval(W_n)=\rho_n^n r_n=r_n.$
\end{proof}

\begin{theorem}[Zaks path in prefix form]
\label{thm:zaks-path-prefix}
For every $n\ge2$, the word $W_n$ generates a Hamiltonian path in $\Pn$ from $\id$ to $r_n$. For $n\ge3$, adding the edge $\{r_n,\id\}$ gives a Hamiltonian cycle in $\Pn$.
\end{theorem}

\begin{proof}
Let
\[
\Phi:S_n\to S_n,\qquad \Phi(\pi)=J_n\pi J_n.
\]
By Lemma \ref{lem:prefix-suffix-conjugacy}, if
\(\pi_{i+1}=\pi_i t_k\) in the suffix-reversal Cayley graph, then $\Phi(\pi_{i+1}) = \Phi(\pi_i)r_k.$
Moreover \(\Phi(\id)=\id\). Hence the \(\Phi\)-image of Zaks' suffix-reversal
Hamiltonian path is a prefix-reversal Hamiltonian path whose transition word is
exactly \(W_n\). By Theorem \ref{thm:zaks-original}, this path visits every
element of \(S_n\) exactly once. Its terminal vertex is
\(\eval(W_n)=r_n\) by Lemma \ref{lem:endpoint}. Therefore \(W_n\)
generates a Hamiltonian path in \(\Pn\) from \(\id\) to \(r_n\).
If \(n\ge3\), then \(r_n\in R_n\), so \(\{r_n,\id\}\) is an edge of \(\Pn\).
Adding this edge gives a Hamiltonian cycle.
\end{proof}

\begin{definition}[The Zaks cycle]
For \(n\ge3\), let \(Z_n\) denote the Hamiltonian cycle obtained from the path generated by \(W_n\) by adding the closing edge \(\{r_n,\id\}\).
\end{definition}

\begin{remark}
The recursion also gives \(|W_n|=n!-1\): if \(|W_{n-1}|=(n-1)!-1\), then
\[
|W_n|
=
(n-1)\bigl(|W_{n-1}|+1\bigr)+|W_{n-1}|
=
n!-1.
\]
This agrees with the length asserted in Zaks' construction.
\end{remark}

\begin{remark}
The graph map \(\Phi(\pi)=J_n\pi J_n\) used in Theorem \ref{thm:zaks-path-prefix} is not asserted to be an automorphism of \(\Pn\). It is a bijection from the suffix-reversal Cayley graph to the prefix-reversal Cayley graph. The stabilizer computations below take place entirely inside the prefix-reversal graph \(\Pn\).
\end{remark}

% ==================================================
\section{The Visible Dihedral Symmetries}
\label{sec:dihedral}

We now prove the inclusion
\[
\langle L_\rho,L_{r_n}\rangle\le \Stab_{L(S_n)}(Z_n),
\qquad
\rho=r_{n-1}r_n.
\]
The two generators arise from distinct features of the recursive word \(W_n\): the element \(L_\rho\) rotates the recursive blocks, while \(L_{r_n}\) reverses the Hamiltonian path. Both assertions are consequences of identities at the level of words.
Throughout this section we write \(\rho=\rho_n=r_{n-1}r_n\). By Lemma \ref{lem:rho-cycle-early},
\[
\rho=[n,1,2,\ldots,n-1],
\]
and \(\ord(\rho)=n\).

\begin{lemma}[Palindromy]
\label{lem:palindromy}
For every \(n\ge2\), \(W_n^R=W_n\).
\end{lemma}

\begin{proof}
The assertion is immediate for \(n=2\). Suppose \(W_{n-1}^R=W_{n-1}\). Put \(A=W_{n-1}\) and \(b=r_n\). Then \(W_n=(Ab)^{n-1}A=A(bA)^{n-1}\). Taking the reversal of this word gives \(W_n^R = A^R(bA^R)^{n-1} = A(bA)^{n-1} = W_n\).
\end{proof}

The next lemma is independent of the pancake graph. It records the elementary cancellation principle behind the reflection \(L_{r_n}\).

\begin{lemma}[Reflection from a palindromic involutory word]
\label{lem:palindromic-reflection}
Let \(W=s_1s_2\cdots s_m\) be a word in involutions of a group \(G\), and assume \(W^R=W\). Let \(q=\eval(W)=s_1s_2\cdots s_m\), and let \(u_0=\id,\ u_i=s_1s_2\cdots s_i\quad(1\le i\le m)\) be the vertex sequence generated by \(W\). Then \(q\,u_i=u_{m-i} \quad (0\le i\le m)\).
\end{lemma}

\begin{proof}
Since \(W^R=W\), the word \(s_1\cdots s_i\) is equal to the reversed terminal segment
\[
s_m s_{m-1}\cdots s_{m-i+1}.
\]
Therefore
\[
q\,u_i
=
(s_1\cdots s_m)(s_1\cdots s_i)
=
(s_1\cdots s_m)(s_m s_{m-1}\cdots s_{m-i+1}).
\]
The factors \(s_m,s_{m-1},\ldots,s_{m-i+1}\) cancel successively because each \(s_j\) is an involution. Hence
\[
q\,u_i=s_1\cdots s_{m-i}=u_{m-i}.
\]
\end{proof}

\begin{proposition}[Reflection symmetry]
\label{prop:reflection}
For every \(n\ge3\), \(L_{r_n}(Z_n)=Z_n\). More precisely, if \(u_0=\id,u_1,\ldots,u_{n!-1}=r_n\) is the Hamiltonian path generated by \(W_n\), then \(L_{r_n}(u_i)=u_{n!-1-i} \quad (0\le i\le n!-1)\).
\end{proposition}

\begin{proof}
By Lemma \ref{lem:palindromy}, \(W_n\) is palindromic. By Lemma \ref{lem:endpoint}, \(\eval(W_n)=r_n\). Applying Lemma \ref{lem:palindromic-reflection} to the word \(W_n\) gives \(r_nu_i=u_{n!-1-i}\) for every \(i\). Thus \(L_{r_n}\) reverses the Hamiltonian path generated by \(W_n\). The closing edge of \(Z_n\) is \(\{r_n,\id\}\); under \(L_{r_n}\), this edge is fixed setwise. Hence \(L_{r_n}\) preserves the closed Hamiltonian cycle \(Z_n\).
\end{proof}

We now turn to the rotational symmetry. Let \(M=(n-1)!\). Let \(v_0=\id,\ v_1,\ldots,v_{M-1}=r_{n-1}\) be the vertex sequence in the embedded copy of \(S_{n-1}\le S_n\) generated by \(W_{n-1}\). The equality \(v_{M-1}=r_{n-1}\) is Lemma \ref{lem:endpoint} applied in rank \(n-1\). The recursion \(W_n=(W_{n-1}r_n)^{n-1}W_{n-1}\) decomposes the Hamiltonian path generated by \(W_n\) into \(n\) ordered blocks \(P_0,P_1,\ldots,P_{n-1}\), each containing \(M\) vertices. Define these blocks by
\[
P_k=(u_{k,0},u_{k,1},\ldots,u_{k,M-1}),
\qquad
u_{k,j}=\rho^k v_j,
\]
for \(0\le k\le n-1,\ 0\le j\le M-1\).

\begin{lemma}[Block decomposition]
\label{lem:block-decomposition}
The ordered vertex sequence of the Hamiltonian path generated by \(W_n\) is \(P_0,P_1,\ldots,P_{n-1}\), where
\[
P_k=(\rho^k v_0,\rho^k v_1,\ldots,\rho^k v_{M-1})
\qquad(0\le k\le n-1).
\]
The internal edge-label word of each block is \(W_{n-1}\). The edge from the last vertex of \(P_k\) to the first vertex of \(P_{k+1}\), for \(0\le k\le n-2\), has label \(r_n\). The closing edge from the last vertex of \(P_{n-1}\) to the first vertex of \(P_0\) also has label \(r_n\).
\end{lemma}

\begin{proof}
After \(k\) full copies of the word \(W_{n-1}r_n\), the accumulated group element is
\[
\eval((W_{n-1}r_n)^k)
=
(r_{n-1}r_n)^k
=
\rho^k.
\]
Thus the \(k\)-th copy of \(W_{n-1}\) begins at \(\rho^k\). Since right multiplication by the letters of \(W_{n-1}\) gives the sequence \(v_0,v_1,\ldots,v_{M-1}\) from the identity, the corresponding block beginning at \(\rho^k\) is \(\rho^k v_0,\rho^k v_1,\ldots,\rho^k v_{M-1}\). This proves the asserted block form and the assertion about internal labels.

The last vertex of \(P_k\) is \(u_{k,M-1}=\rho^k r_{n-1}\). Multiplying on the right by \(r_n\) gives \(\rho^k r_{n-1}r_n = \rho^k\rho = \rho^{k+1}\). For \(0\le k\le n-2\), this is the first vertex of \(P_{k+1}\), and the connecting edge has label \(r_n\). For \(k=n-1\), the same computation gives \(\rho^{n-1}r_{n-1}r_n = \rho^n = \id\), which is the first vertex of \(P_0\). This is precisely the closing edge of \(Z_n\).
\end{proof}

\begin{proposition}[Rotational symmetry]
\label{prop:rotation}
For every \(n\ge3\), \(L_\rho(Z_n)=Z_n\). More precisely, \(L_\rho\) sends \(P_k\) to \(P_{k+1}\), with indices taken modulo \(n\), and therefore acts on the cyclic ordering of \(Z_n\) as a shift by \((n-1)!\) vertices.
\end{proposition}

\begin{proof}
For every \(k,j\), \(L_\rho(u_{k,j}) = \rho(\rho^k v_j) = \rho^{k+1}v_j = u_{k+1,j}\), where \(k+1\) is read modulo \(n\). Thus \(L_\rho\) sends the ordered block \(P_k\) to the ordered block \(P_{k+1}\). By Lemma \ref{lem:block-decomposition}, this preserves all internal block edges, all connecting edges, and the closing edge. Hence \(L_\rho\) preserves \(Z_n\). Since each block has \((n-1)!\) vertices, the induced action on the cyclic vertex ordering of \(Z_n\) is a shift by \((n-1)!\) positions.
\end{proof}

\begin{lemma}
\label{lem:dihedral-relations}
For every \(n\ge3\), \(\langle L_\rho,L_{r_n}\rangle\cong D_n\), where \(D_n\) denotes the dihedral group of order \(2n\).
\end{lemma}

\begin{proof}
By Lemma \ref{lem:rho-cycle-early}, \(L_\rho\) has order \(n\), and \(L_{r_n}\) has order \(2\). Moreover, \(L_{r_n}L_\rho L_{r_n} = L_{r_n\rho r_n}\). Since
\[
r_n\rho r_n
=
r_n(r_{n-1}r_n)r_n
=
r_nr_{n-1}
=
(r_{n-1}r_n)^{-1}
=
\rho^{-1},
\]
we have \(L_{r_n}L_\rho L_{r_n}=L_{\rho^{-1}}\). Thus \(\langle L_\rho,L_{r_n}\rangle\) is a quotient of the dihedral group of order \(2n\). It remains only to note that \(L_{r_n}\notin\langle L_\rho\rangle\). Equivalently, \(r_n\notin\langle \rho\rangle\). If \(\rho^k=r_n\), then comparing the image of \(1\) gives \(k\equiv1\pmod n\), because \(\rho^k(1)=r_n(1)=n\). But then \(\rho^k=\rho\), while \(\rho(2)=1 \text{ and } r_n(2)=n-1\), which are unequal for \(n\ge3\). Hence \(\langle L_\rho\rangle\) and \(L_{r_n}\langle L_\rho\rangle\) are distinct cosets, and the generated group has order \(2n\).
\end{proof}

Combining Propositions \ref{prop:reflection} and \ref{prop:rotation} with Lemma \ref{lem:dihedral-relations} gives the lower bound needed for the main theorem.

\begin{corollary}[Visible dihedral subgroup]
\label{cor:visible-dihedral}
For every \(n\ge3\), \(\langle L_\rho,L_{r_n}\rangle \le \Stab_{L(S_n)}(Z_n)\), and this subgroup is isomorphic to \(D_n\).
\end{corollary}

% ==================================================
\subsection{The cyclic label word}
\label{subsec:cyclic-label-word}

Let $\calT_n=W_nr_n$ be the cyclic edge-label word of $Z_n$, based at $\id$. From $W_n=(W_{n-1}r_n)^{n-1}W_{n-1}$ we obtain $\calT_n=(W_{n-1}r_n)^n$. Thus the label word is $n$-periodic at the block level, with fundamental block $W_{n-1}r_n$ of length $(n-1)!$. This is the label-level form of the block rotation in Proposition~\ref{prop:rotation}.

\begin{example}[Rank $4$]
\label{ex:rank-four-label-word}
For $n=4$, we have $W_3=r_2r_3r_2r_3r_2$, and hence $\calT_4=(W_3r_4)^4=(r_2r_3r_2r_3r_2r_4)^4$. Thus the cycle has four blocks, each with six vertices. The rotation is induced by $\rho=r_3r_4=[4,1,2,3]$, and $L_\rho$ advances the cycle by one block. The reflection $L_{r_4}$ reverses the Hamiltonian path. Consequently the visible left-regular stabilizer is $\langle L_{r_3r_4},L_{r_4}\rangle\cong D_4$.

This explicit label word will be used in Section~\ref{sec:automorphisms} to exclude the exceptional non-left automorphism of $\mathcal P_4$. In $\calT_4$, every occurrence of $r_4$ is flanked on both sides by $r_2$.
\end{example}

% ==================================================
\section{Cyclic-Order Rigidity and the Left-Regular Stabilizer}
\label{sec:rigidity}

The lower bound in Corollary \ref{cor:visible-dihedral} uses the particular recursion defining \(W_n\). The reverse inclusion is a general fact about cyclic orderings of finite groups. We isolate it in a form independent of the pancake graph.

Let \(X\) be a finite set with \(|X|=N\ge3\). An oriented cyclic ordering of \(X\) is an equivalence class of ordered \(N\)-tuples
\[
(x_0,x_1,\ldots,x_{N-1})
\]
under cyclic rotation. Its reversal is the opposite oriented cyclic ordering. In this section, a cyclic order means the corresponding unoriented object, i.e. the equivalence class also modulo reversal. Equivalently, a cyclic order is a Hamiltonian cycle on vertex set \(X\).

The automorphism group of such an unoriented cyclic order is the dihedral group of order \(2N\). Its orientation-preserving elements are the cyclic shifts, and its orientation-reversing elements are the reflections. A permutation of \(X\) preserves the cyclic order if and only if it acts as an automorphism of this Hamiltonian cycle.

For \(a\in G\), put
\[
N_G^{\pm}(a)=\{g\in G:gag^{-1}\in\{a,a^{-1}\}\}.
\]
We call this subgroup the dihedral normalizer of \(a\).

\begin{theorem}[Left-regular rigidity from one rotation]
\label{thm:cyclic-order-rigidity}
Let \(G\) be a finite group, let \(a\in G\) have order at least \(3\), and let
\(C\) be an unoriented cyclic order on the set \(G\). If \(L_a\) preserves
\(C\), then
\[
\Stab_{L(G)}(C)\le L(N_G^{\pm}(a)).
\]
Equivalently, if \(L_b\) preserves \(C\), then
\[
bab^{-1}\in\{a,a^{-1}\}.
\]
\end{theorem}
\begin{proof}
Let \(N=|G|\). The permutations of \(G\) preserving \(C\) form a dihedral group acting on the \(N\) vertices of the cycle. Since \(\ord(L_a)=\ord(a)\ge3\), \(L_a\) cannot be a reflection; every reflection in a finite cycle has order \(2\). Thus \(L_a\) is a rotation of \(C\).

If \(L_b\) also preserves \(C\), then \(L_b\) is either a rotation or a reflection of the same cycle. In the dihedral group of a cycle, rotations commute with all rotations, while reflections invert every rotation. Hence $L_bL_aL_b^{-1}\in\{L_a,L_a^{-1}\}.$
The left regular action is faithful, and therefore $bab^{-1}\in\{a,a^{-1}\}.$
\end{proof}

\begin{remark}
The subgroup \(N_G^{\pm}(a)\) is generally smaller than the full normalizer
\(N_G(\langle a\rangle)\). Indeed,
\[
N_G^{\pm}(a)
=
C_G(a)\cup\{g\in G:gag^{-1}=a^{-1}\},
\]
where \(C_G(a)\) denotes the centralizer of \(a\). Thus \(N_G^{\pm}(a)\)
records only those conjugations that preserve the cyclic order generated by
\(a\), up to reversal.

In particular, if \(a\in S_n\) is an \(n\)-cycle, then
\[
|N_{S_n}(\langle a\rangle)|=n\varphi(n),
\qquad
|N_{S_n}^{\pm}(a)|=2n
\quad(n\ge3).
\]
\end{remark}

\begin{lemma}
\label{lem:dihedral-normalizer-rho}
Let \(\rho=[n,1,2,\ldots,n-1]\in S_n, \quad n\ge3\). Then \(|N_{S_n}^{\pm}(\rho)|=2n\). More precisely, \(N_{S_n}^{\pm}(\rho)=C_{S_n}(\rho)\sqcup r_nC_{S_n}(\rho)\), and \(C_{S_n}(\rho)=\langle\rho\rangle\).
\end{lemma}

\begin{proof}
Since \(\rho\) is an \(n\)-cycle, its centralizer in \(S_n\) is the cyclic group it generates: \(C_{S_n}(\rho)=\langle\rho\rangle\), of order \(n\). By Lemma \ref{lem:dihedral-relations}, \(r_n\rho r_n=\rho^{-1}\). Thus \(r_nC_{S_n}(\rho) \subseteq \{g\in S_n:g\rho g^{-1}=\rho^{-1}\}\). Conversely, if \(g\rho g^{-1}=\rho^{-1}\), then \(r_ng\rho g^{-1}r_n = r_n\rho^{-1}r_n = \rho\). Hence \(r_ng\in C_{S_n}(\rho)\), so \(g\in r_nC_{S_n}(\rho)\). Therefore \(\{g\in S_n:g\rho g^{-1}=\rho^{-1}\} = r_nC_{S_n}(\rho)\). The two cosets \(C_{S_n}(\rho) \quad\text{and}\quad r_nC_{S_n}(\rho)\) are disjoint. If they were not, then some element of \(C_{S_n}(\rho)\) would conjugate \(\rho\) to \(\rho^{-1}\). But every element of \(C_{S_n}(\rho)\) commutes with \(\rho\), so this would imply \(\rho=\rho^{-1}\), contradicting \(\ord(\rho)=n\ge3\). Hence \(|N_{S_n}^{\pm}(\rho)| = |C_{S_n}(\rho)|+|r_nC_{S_n}(\rho)| = 2n\).
\end{proof}

We now apply the preceding lemma to the cyclic ordering defined by \(Z_n\). Recall that \(\Stab_{L(S_n)}(Z_n) = \{L_g\in L(S_n):L_g(Z_n)=Z_n\}\). Equivalently, after identifying \(L(S_n)\) with \(S_n\), this stabilizer consists of all \(g\in S_n\) for which left multiplication by \(g\) preserves the edge set, and hence the cyclic ordering, of the Hamiltonian cycle \(Z_n\).

\begin{theorem}[Exact left-regular stabilizer]
\label{thm:left-stabilizer}
For every \(n\ge3\), \(\Stab_{L(S_n)}(Z_n) = \langle L_\rho,L_{r_n}\rangle \cong D_n\), where \(\rho=r_{n-1}r_n=[n,1,2,\ldots,n-1]\).
\end{theorem}

\begin{proof}
By Corollary \ref{cor:visible-dihedral}, \(\langle L_\rho,L_{r_n}\rangle \le \Stab_{L(S_n)}(Z_n)\), and this subgroup has order \(2n\). It remains to prove that no further left translations preserve \(Z_n\). Let \(L_g\in\Stab_{L(S_n)}(Z_n)\). By Proposition \ref{prop:rotation}, \(L_\rho\) preserves the cyclic ordering of \(Z_n\). Since \(\ord(\rho)=n\ge3\), Theorem \ref{thm:cyclic-order-rigidity}, applied to the group \(S_n\), the element \(\rho\), and the cyclic ordering induced by \(Z_n\), gives \[ g\rho g^{-1}\in\{\rho,\rho^{-1}\}. \] Thus \(g\in N_{S_n}^{\pm}(\rho)\). By Lemma \ref{lem:dihedral-normalizer-rho}, \(|N_{S_n}^{\pm}(\rho)|=2n\). Therefore \(|\Stab_{L(S_n)}(Z_n)|\le 2n\). The visible subgroup already has order \(2n\), so equality holds: \(\Stab_{L(S_n)}(Z_n) = \langle L_\rho,L_{r_n}\rangle\). The isomorphism with \(D_n\) is Lemma \ref{lem:dihedral-relations}.
\end{proof}

The proof actually gives the following more precise description of the two cosets in the stabilizer.

\begin{corollary}
\label{cor:two-cosets}
For every \(n\ge3\),
\[
\Stab_{L(S_n)}(Z_n)
=
\{L_{\rho^i}:0\le i<n\}
\sqcup
\{L_{r_n\rho^i}:0\le i<n\}.
\]
The first coset consists of the orientation-preserving symmetries of the cyclic ordering of \(Z_n\), and the second consists of the orientation-reversing symmetries.
\end{corollary}

\begin{proof}
The equality as a set follows from
\[
\langle L_\rho,L_{r_n}\rangle
=
\langle L_\rho\rangle
\sqcup
L_{r_n}\langle L_\rho\rangle
\]
and from Theorem \ref{thm:left-stabilizer}.

By Proposition \ref{prop:rotation}, \(L_\rho\) preserves the orientation of the cyclic ordering and acts as a rotation. Hence every power of \(L_\rho\) is orientation-preserving. By Proposition \ref{prop:reflection}, \(L_{r_n}\) reverses the cyclic ordering. Therefore every element in the coset \(L_{r_n}\langle L_\rho\rangle\) is orientation-reversing. These are all elements of the stabilizer by the preceding equality.
\end{proof}

\begin{remark}
The upper bound in Theorem \ref{thm:left-stabilizer} is independent of the pancake generating set. Once a Hamiltonian cycle on \(S_n\) is invariant under left multiplication by the \(n\)-cycle \(\rho\), every left-regular stabilizer element lies in the dihedral normalizer
\[
N_{S_n}^{\pm}(\rho)=\{g\in S_n:g\rho g^{-1}\in\{\rho,\rho^{-1}\}\}.
\]
This subgroup has order \(2n\), whereas the full normalizer of \(\langle\rho\rangle\) has order \(n\varphi(n)\).
\end{remark}

It remains to compare this subgroup with the full graph-automorphism stabilizer of \(Z_n\). For this we use the automorphism theorem for pancake graphs.

% ==================================================
\section{Full Automorphism Stabilizers and Hamilton Compression}
\label{sec:automorphisms}

For \(n\ge5\), the pancake graph has no graph automorphisms beyond the left
regular action. Thus Theorem \ref{thm:left-stabilizer} already gives the full
stabilizer in these ranks. The exceptional ranks \(3\) and \(4\) require
separate treatment.
\begin{theorem}[Deng--Zhang {\cite[Theorem 3.2]{deng-zhang}}]
\label{thm:deng-zhang}
For \(n\ge5\),
\[
\Aut(\Pn)=L(S_n).
\]
\end{theorem}

\begin{proposition}[Rank \(4\)]
\label{prop:aut-p4}
For \(n=4\),
\[
\Aut(\mathcal P_4)=L(S_4)\rtimes\langle\alpha\rangle,
\]
where \(\alpha\) is conjugation by \((2\,3)\). Equivalently, $\alpha(r_2)=r_3, \alpha(r_3)=r_2, \alpha(r_4)=r_4.$
\end{proposition}

\begin{proof}
Deng and Zhang record that \(|\Aut(\mathcal P_4)|=48\). Since \(L(S_4)\le \Aut(\mathcal P_4)\) acts regularly on the \(24\) vertices
of \(\mathcal P_4\), the full automorphism group is transitive on vertices.
Hence the orbit of \(\id\) under \(\Aut(\mathcal P_4)\) has size \(24\).
By orbit--stabilizer and \(|\Aut(\mathcal P_4)|=48\), the stabilizer of
\(\id\) in \(\Aut(\mathcal P_4)\) has order \(2\).

Conjugation by \((2\,3)\) maps the generating set
\(\{r_2,r_3,r_4\}\) to itself; explicitly,
\[
r_2\mapsto r_3,\qquad r_3\mapsto r_2,\qquad r_4\mapsto r_4.
\]
It therefore induces a graph automorphism \(\alpha\) of
\(\Cay(S_4,\{r_2,r_3,r_4\})\) fixing \(\id\). Since \(\alpha\ne 1\), the
identity stabilizer is exactly \(\{1,\alpha\}\).

Every automorphism of \(\mathcal P_4\) is therefore uniquely of the form
\(L_g\) or \(L_g\alpha\), with \(g\in S_4\). Finally, \(\alpha\) normalizes
\(L(S_4)\), since $\alpha L_g\alpha^{-1}=L_{\alpha(g)}.$
Hence $\Aut(\mathcal P_4)=L(S_4)\rtimes\langle\alpha\rangle.$
\end{proof}

\begin{remark}
For \(n\ge5\), the graph \(\mathcal P_n\) is a graphical regular representation
of \(S_n\). In rank \(4\), the only nontrivial automorphism fixing \(\id\)
is the automorphism \(\alpha\) induced by conjugation with \((2\,3)\); it
interchanges \(r_2\) and \(r_3\) and fixes \(r_4\). Thus the non-left
automorphisms of \(\mathcal P_4\) are precisely the elements of the coset
\(L(S_4)\alpha\).
\end{remark}

It remains only to exclude the non-left coset in rank \(4\). We use the cyclic edge-label word \[ \calT_4=(r_2r_3r_2r_3r_2r_4)^4 \] from Example \ref{ex:rank-four-label-word}. In this cyclic word every occurrence of \(r_4\) is flanked, on both sides, by \(r_2\). Because the edge labels are intrinsic and the generators are involutions, this local cyclic property is unaffected by changing the initial vertex or by reversing the orientation of the traversal.

\begin{lemma}
\label{lem:alpha-not-stabilize}
No automorphism in the coset \(L(S_4)\alpha\) stabilizes \(Z_4\).
\end{lemma}

\begin{proof}
Let \(F=L_g\alpha\), with \(g\in S_4\). Left multiplication preserves the
intrinsic edge labels, while \(\alpha\) sends each label \(r_i\) to
\(\alpha(r_i)\). Hence the cyclic label word of \(F(Z_4)\), up to cyclic
rotation and reversal, is obtained from \(\calT_4\) by applying \(\alpha\) to
each letter:
\[
\alpha(\calT_4)=(r_3r_2r_3r_2r_3r_4)^4.
\]
In this cyclic word every occurrence of \(r_4\) is flanked by \(r_3\) on both
sides. In contrast, in
\[
\calT_4=(r_2r_3r_2r_3r_2r_4)^4
\]
every occurrence of \(r_4\) is flanked by \(r_2\) on both sides, and this
remains true after reversal. Since \(r_2\ne r_3\), the labeled cyclic orders
cannot coincide. Therefore \(F(Z_4)\ne Z_4\).
\end{proof}

We can now state the full graph-automorphism stabilizer of \(Z_n\).

\begin{theorem}[Full graph-automorphism stabilizer]
\label{thm:full-stabilizer}
For \(n\ge4\), \(\Stab_{\Aut(\Pn)}(Z_n) = \langle L_\rho,L_{r_n}\rangle \cong D_n\), where \(\rho=r_{n-1}r_n\). For \(n=3\), the graph \(\mathcal P_3\) is the \(6\)-cycle \(Z_3\), and hence \(\Stab_{\Aut(\mathcal P_3)}(Z_3)=\Aut(\mathcal P_3)\cong D_6\), the dihedral group of order \(12\).
\end{theorem}
\begin{proof}
If \(n\ge5\), then Theorem \ref{thm:deng-zhang} gives \(\Aut(\Pn)=L(S_n)\), and the assertion follows from Theorem \ref{thm:left-stabilizer}.

Let \(n=4\). By Proposition \ref{prop:aut-p4},
\[
\Aut(\mathcal P_4)=L(S_4)\rtimes\langle\alpha\rangle.
\]
The intersection of \(\Stab_{\Aut(\mathcal P_4)}(Z_4)\) with \(L(S_4)\) is, by Theorem \ref{thm:left-stabilizer},
\[
\langle L_{r_3r_4},L_{r_4}\rangle\cong D_4.
\]
Lemma \ref{lem:alpha-not-stabilize} excludes the coset \(L(S_4)\alpha\). Hence
\[
\Stab_{\Aut(\mathcal P_4)}(Z_4)
=
\langle L_{r_3r_4},L_{r_4}\rangle.
\]

Finally, for \(n=3\), the graph \(\mathcal P_3=\Cay(S_3,\{r_2,r_3\})\) is a cycle of length \(6\), because the cyclic transition word is \((r_2r_3)^3\). Thus \(Z_3\) is the entire graph \(\mathcal P_3\), and its stabilizer in the full graph automorphism group is all of \(\Aut(\mathcal P_3)\), namely the dihedral group of the \(6\)-cycle.
\end{proof}

\begin{remark}
The distinction between the left-regular and full graph-automorphism
stabilizers occurs only in low rank. The situation is summarized as follows:
\[
\begin{array}{c|c|c|c|c}
n & |V(\mathcal P_n)| & \ord(\rho) &
|\Stab_{L(S_n)}(Z_n)| &
|\Stab_{\Aut(\mathcal P_n)}(Z_n)|\\
\hline
3 & 6   & 3 & 6  & 12\\
4 & 24  & 4 & 8  & 8\\
5 & 120 & 5 & 10 & 10
\end{array}
\]
For \(n=3\), \(\mathcal P_3\) is itself a \(6\)-cycle. For \(n=4\), the graph
has a non-left automorphism coset, but Lemma \ref{lem:alpha-not-stabilize}
shows that this coset does not preserve \(Z_4\).
\end{remark}

% --------------------------------------------------
\subsection{Hamilton compression of the Zaks cycle}

We use the terminology of Gregor, Merino, and M\"utze
\cite{gregor-merino-mutze}. Let \(G\) be a graph on \(N\) vertices, and let
\[
C=(x_0,x_1,\ldots,x_{N-1})
\]
be a Hamiltonian cycle, with indices read modulo \(N\). For \(k\mid N\), the
cycle \(C\) is \(k\)-symmetric if the map
\[
x_i\longmapsto x_{i+N/k}
\qquad(0\le i<N)
\]
is induced by an automorphism of \(G\). Its compression factor is
\[
\Comp(G,C)=\max\{k:C\text{ is }k\text{-symmetric}\},
\]
and the Hamilton compression of \(G\) is
\[
\Comp(G)=\max_C\Comp(G,C),
\]
where \(C\) ranges over all Hamiltonian cycles of \(G\).

If \(C\) is \(k\)-symmetric, then the inducing automorphism is orientation-preserving on \(C\) and has order \(k\) on \(V(C)=V(G)\). Hence \(\Comp(G,C)\) is bounded above by the maximum order of an orientation-preserving element of \(\Stab_{\Aut(G)}(C)\).

\begin{theorem}[Compression factor of the Zaks cycle]
\label{thm:zaks-compression}
For the Zaks cycle \(Z_n\) in the pancake graph \(\Pn\),
\[
\Comp(\Pn,Z_n)=
\begin{cases}
6, & n=3,\\
n, & n\ge4.
\end{cases}
\]
\end{theorem}

\begin{proof}
Let $N=n!$. For $n\ge4$, Theorem~\ref{thm:full-stabilizer} gives $\Stab_{\Aut(\Pn)}(Z_n)=\langle L_\rho,L_{r_n}\rangle$. By Corollary~\ref{cor:two-cosets}, the orientation-preserving subgroup of this stabilizer is exactly $\langle L_\rho\rangle$, of order $n$. Proposition~\ref{prop:rotation} shows that $L_\rho$ shifts the cyclic ordering of $Z_n$ by $(n-1)!=N/n$ vertices. Thus $Z_n$ is $n$-symmetric, and $\Comp(\Pn,Z_n)\ge n$.

Conversely, suppose $Z_n$ is $k$-symmetric. Then the defining shift $x_i\mapsto x_{i+N/k}$ is induced by an orientation-preserving automorphism of $Z_n$ of order $k$. This automorphism lies in the orientation-preserving subgroup $\langle L_\rho\rangle$, whose order is $n$. Therefore $k\le n$. Hence $\Comp(\Pn,Z_n)=n$ for $n\ge4$.

For $n=3$, the graph $\mathcal P_3$ is the $6$-cycle $Z_3$. The rotation by one vertex is a graph automorphism, so $Z_3$ is $6$-symmetric. Since no Hamiltonian cycle on six vertices has compression factor exceeding $6$, we get $\Comp(\mathcal P_3,Z_3)=6$.

\end{proof}

\begin{corollary}
\label{cor:compression-lower-bound}
For every \(n\ge4\), \(\Comp(\Pn)\ge n\).
\end{corollary}

\begin{proof}
The graph parameter \(\Comp(\Pn)\) is the maximum of \(\Comp(\Pn,C)\) over all Hamiltonian cycles \(C\) in \(\Pn\). The Zaks cycle is one such Hamiltonian cycle, and Theorem \ref{thm:zaks-compression} gives \(\Comp(\Pn,Z_n)=n\).
\end{proof}

\begin{remark}
Theorem \ref{thm:zaks-compression} determines the compression factor of the Zaks cycle, not the Hamilton compression of the pancake graph itself. It remains possible that \(\Pn\) admits Hamiltonian cycles with compression factor strictly larger than \(n\). The stabilizer computation proves only that no such improvement can occur within the automorphism stabilizer of the particular cycle \(Z_n\).
\end{remark}

% ==================================================
\section{Consequences and Further Directions}
\label{sec:further}

We record three consequences of the proof. First, the block rotation has a
quotient-voltage interpretation. Second, Zaks' ranking and unranking procedures
transfer directly to the prefix-reversal form. Third, the stabilizer computation
suggests several classification problems for symmetric Hamiltonian cycles in
pancake graphs and related Cayley graphs.

% --------------------------------------------------
\subsection{The quotient by the block rotation}

We use the standard quotient viewpoint for graphs with a semiregular automorphism, as in Alspach's framework for lifting Hamilton cycles from quotient graphs \cite{alspach-lifting}. Let $A=\langle L_\rho\rangle.$
Since \(\rho\) has order \(n\) and the left regular action is free, \(A\) acts semiregularly on \(V(\Pn)=S_n\). We write \(\Pn/A\) for the quotient multigraph whose vertices are the \(A\)-orbits and whose edges are induced by edges of \(\Pn\). Multiple edges may occur.

Let \(H_n=\{\pi\in S_n:\pi(n)=n\}\). This is the standard embedded copy of \(S_{n-1}\) in \(S_n\), generated by \(r_2,\ldots,r_{n-1}\). Every \(A\)-orbit meets \(H_n\) in exactly one vertex. Indeed, if \(\pi\in S_n\), then the final entry of \(\rho^i\pi\) is \((\rho^i\pi)(n)=\rho^i(\pi(n))\). Since \(\rho\) is an \(n\)-cycle on the values \(\{1,\ldots,n\}\), there is a unique \(i\pmod n\) such that \(\rho^i(\pi(n))=n\). Thus quotient vertices may be indexed by elements of \(H_n\). We denote the orbit of \(h\in H_n\) by \([h]=\{\rho^i h:0\le i<n\}\). The quotient carries a natural voltage description. For \(h\in H_n\) and \(s\in R_n\), there are unique \(\nu(h,s)\in\mathbb Z/n\mathbb Z \text{ and } h'\in H_n\) such that \(hs=\rho^{\nu(h,s)}h'\). The corresponding quotient edge joins \([h]\) to \([h']\), has label \(s\), and has voltage \(\nu(h,s)\). With this convention, the lift of a quotient edge starting at \(\rho^i h\) ends at \(\rho^i hs = \rho^{i+\nu(h,s)}h'\). Let \(v_0=\id,\ v_1,\ldots,v_{M-1}=r_{n-1}\), \(M=(n-1)!\), be the vertex
sequence in \(H_n\) generated by \(W_{n-1}\). Since
\(v_0,\ldots,v_{M-1}\) are distinct and \(H_n\) is a transversal for the
\(A\)-orbits, the quotient vertices \([v_0],\ldots,[v_{M-1}]\) are distinct.
The quotient image of \(Z_n\) is the cycle $[v_0],[v_1],\ldots,[v_{M-1}],[v_0]$
in \(\Pn/A\).

\begin{proposition}[Voltage form of the Zaks lift]
\label{prop:voltage-lift}
Under the identification of \(V(\Pn/A)\) with \(H_n\), the quotient image of \(Z_n\) is the closed walk with label word \(W_{n-1}r_n\). All edges belonging to the subword \(W_{n-1}\) have voltage \(0\), and the final edge labelled \(r_n\) has voltage \(1\). Consequently the total voltage of the quotient cycle is \(1\in\mathbb Z/n\mathbb Z\), and its lift is the single Hamiltonian cycle \(Z_n\).
\end{proposition}

\begin{proof}
For every letter \(r_k\) with \(2\le k\le n-1\), right multiplication by \(r_k\) preserves \(H_n\). Hence if \(v_{j+1}=v_jr_k\), then \(v_jr_k=\rho^0v_{j+1}\), so the corresponding quotient edge has voltage \(0\). The final vertex of the path generated by \(W_{n-1}\) is \(v_{M-1}=r_{n-1}\). Multiplication by \(r_n\) gives \(r_{n-1}r_n=\rho=\rho^1\id\). Thus the closing quotient edge from \([r_{n-1}]\) to \([\id]\) has label \(r_n\) and voltage \(1\). Starting from \(\id\), one traversal of the quotient cycle lifts to the block \(P_0=(v_0,\ldots,v_{M-1})\) and ends in the next fibre, at \(\rho\). The next traversal lifts to \(P_1\), and in general the \(i\)-th traversal lifts to \(P_i\). Since the total voltage is \(1\), which generates \(\mathbb Z/n\mathbb Z\), the lift closes only after \(n\) quotient traversals. The resulting lifted cycle is precisely \(P_0,P_1,\ldots,P_{n-1}\), with the connecting \(r_n\)-edges described in Lemma \ref{lem:block-decomposition}. This is \(Z_n\).
\end{proof}

\begin{remark}
Proposition \ref{prop:voltage-lift} is the quotient counterpart of the block decomposition. It also explains why the compression factor \(n\) appears naturally: the cycle \(Z_n\) is obtained by lifting a quotient cycle on \((n-1)!\) vertices through a cyclic voltage group of order \(n\), with total voltage a generator.
\end{remark}

% --------------------------------------------------
\subsection{Algorithmic transfer from Zaks' ordering}

Zaks' construction was designed as a permutation-generation algorithm, not
merely as an existence proof. The conjugation $\Phi(\pi)=J_n\pi J_n$ transfers generation, ranking, and unranking from the suffix-reversal ordering
to the prefix-reversal path.

\begin{proposition}[Algorithmic transfer]
\label{prop:algorithmic-transfer}
Generation, ranking, and unranking for the prefix Zaks path reduce to the
corresponding operations for Zaks' suffix-reversal order by the fixed conjugacy
\[
\Phi(\pi)=J_n\pi J_n.
\]
Explicitly,
\[\operatorname{rank}_{Z}(\pi)
=
\operatorname{rank}_{\mathrm{suf}}(J_n\pi J_n),
\]
and
\[
\operatorname{unrank}_{Z}(m)
=
J_n\,\operatorname{unrank}_{\mathrm{suf}}(m)\,J_n.
\]
The transition-length sequence is unchanged. In one-line notation, applying
\(\Phi\) costs \(O(n)\) time and requires no search.
\end{proposition}

\begin{proof}
The transition length \(k\) in Zaks' suffix order corresponds, under \(\Phi\), to the prefix transition \(r_k\) by Lemma \ref{lem:prefix-suffix-conjugacy}. Thus the transition-length sequence is identical.

The ranking and unranking formulae above follow from the identity \(\Phi(\sigma t_k)=\Phi(\sigma)r_k\) for every suffix reversal \(t_k\). Since \(\Phi\) is an involution, conversion between the suffix and prefix orderings requires no additional search or recursion; it is the fixed relabelling of positions given by conjugation with \(J_n\).
\end{proof}

\begin{remark}
If permutations are stored explicitly in one-line notation, applying an actual prefix reversal of length \(k\) costs \(O(k)\) elementary swaps. Zaks' transition statistics therefore transfer unchanged to the prefix form. In particular, the average reversed segment length remains bounded by the same constant as in the suffix-reversal formulation. If, instead, one charges unit cost for producing the next transition label, then the label-generation problem is exactly the same for the suffix and prefix versions.
\end{remark}

% --------------------------------------------------
\subsection{Open problems}

The following problems delimit the scope of the present result.

\begin{problem}[Comparison with other pancake cycles]
Kanevsky and Feng constructed Hamiltonian cycles in pancake graphs together
with ranking and unranking algorithms \cite{kanevsky-feng}. Determine whether
their cycles are equivalent to \(Z_n\) under graph automorphisms, reversal of
the Hamiltonian cycle, or relabelling of the generating set. If they are
inequivalent, compute their left-regular stabilizers and compare their
compression factors with that of \(Z_n\).
\end{problem}

\begin{problem}[Coxeter-theoretic analogues]
Let \(W\) be a finite Coxeter group with a nested chain of standard parabolic
subgroups
\[
W_1<W_2<\cdots<W_m=W,
\]
and let \(w_i\) denote the longest element of \(W_i\). For Cayley graphs
generated by \(\{w_2,\ldots,w_m\}\), determine whether recursive Gray-code
constructions analogous to Zaks' construction admit visible left-regular
rotations whose stabilizers are controlled by cyclic-order rigidity.

In type \(A_{n-1}\), the generators \(r_k\) are precisely the longest elements
of the standard parabolic subgroups
\[
S_2<S_3<\cdots<S_n,
\]
and \(r_{n-1}r_n\) is an \(n\)-cycle, hence a Coxeter element.
\end{problem}

\begin{problem}[Hamilton compression of pancake graphs]
Determine
\[
\Comp(\Pn)=\max_C\Comp(\Pn,C),
\]
where \(C\) ranges over all Hamiltonian cycles of \(\Pn\). Theorem
\ref{thm:zaks-compression} gives \(\Comp(\Pn)\ge n\) for \(n\ge4\).

For \(n\ge5\), Deng--Zhang's theorem implies that every automorphism of
\(\Pn\) is a left translation. Hence any \(k\)-symmetric Hamiltonian cycle is
rotated by some element of \(S_n\) of order \(k\). Which conjugacy classes of
\(S_n\) can occur in this way? In particular, can \(\Pn\) admit a Hamiltonian
cycle with compression factor strictly larger than \(n\)?
\end{problem}

\begin{problem}[\(\rho\)-invariant Hamiltonian cycles]
Classify the Hamiltonian cycles \(C\subseteq \Pn\) satisfying
\[
L_\rho(C)=C.
\]
Equivalently, classify the Hamiltonian cycles in the quotient multigraph
\[
\Pn/\langle L_\rho\rangle
\]
whose total voltage generates \(\mathbb Z/n\mathbb Z\).
\end{problem}

\begin{problem}[Stabilizer spectrum]
For \(n\ge5\), determine the set
\[
\Sigma_n=
\left\{
|\Stab_{\Aut(\Pn)}(C)|:
C\text{ is a Hamiltonian cycle of }\Pn
\right\}.
\]
Since \(\Aut(\Pn)=L(S_n)\) for \(n\ge5\), this is equivalently the spectrum of left-regular stabilizer sizes of Hamiltonian cycles in \(\Pn\). Theorem \ref{thm:full-stabilizer} shows that \(2n\in\Sigma_n\).
\end{problem}

\begin{problem}[Other recursive Cayley Gray codes]
Determine whether analogous stabilizer computations hold for recursive Hamiltonian cycles in other standard Cayley graphs of \(S_n\), including the permutahedron, the star graph, and the transposition graph. More generally, identify recursive Gray-code constructions for which a visible left-regular rotation and cyclic-order rigidity determine the full left-regular stabilizer.
\end{problem}

% ==================================================
\section*{Conclusion}

Zaks' recursive Gray code yields a Hamiltonian cycle \(Z_n\) in the pancake
graph whose recursive block structure gives an \(n\)-fold rotation and whose
palindromy gives a reflection. We proved that these visible symmetries exhaust
the left-regular stabilizer:
\[
\Stab_{L(S_n)}(Z_n)
=
\langle L_{r_{n-1}r_n},L_{r_n}\rangle
\cong D_n.
\]
The upper bound is supplied by cyclic-order rigidity: once \(Z_n\) is invariant
under the rotation \(L_{r_{n-1}r_n}\), every left-regular stabilizer element is
forced into the dihedral normalizer of the \(n\)-cycle \(r_{n-1}r_n\).

For \(n\ge5\), Deng--Zhang's automorphism theorem identifies
\(\Aut(\Pn)\) with \(L(S_n)\), so the same dihedral group is the full
graph-automorphism stabilizer. The exceptional ranks \(3\) and \(4\) are
accounted for separately. The same stabilizer computation gives the compression
factor of \(Z_n\): it is \(n\) for \(n\ge4\) and \(6\) for \(n=3\). Finally, the
quotient-voltage formulation realizes \(Z_n\) as the lift of a quotient cycle on
\((n-1)!\) vertices with total voltage generating \(\mathbb Z/n\mathbb Z\).

% ==================================================

\end{document}